\documentclass[journal]{IEEEtran}

\usepackage[utf8]{inputenc}
\usepackage[english]{babel}
\usepackage{amsmath,amsthm,amssymb,bm}

\usepackage[square,numbers,sort,compress]{natbib}
\usepackage{hyperref}
\hypersetup{bookmarks=true, colorlinks=false, hidelinks}
\usepackage{mathtools}
\usepackage{booktabs}
\usepackage{multirow}

\newcommand{\diff}{\mathop{}\!\mathrm{d}}

\newcommand{\cond}{{\;|\;}}

\newcommand{\expecsym}{\operatorname{\mathbb{E}}}     
\newcommand{\covsym}{\operatorname{Cov}}     
\newcommand{\tracesym}{\operatorname{tr}}           
\let\expec\relax
\let\cov\relax
\let\trace\relax
\newcommand{\expec}[1]{\expecsym [#1]}
\newcommand{\expecbig}[1]{\expecsym \big[ #1 \big]}
\newcommand{\cov}[1]{\covsym [#1]}
\newcommand{\trace}[1]{\tracesym (#1)}
\newcommand{\tracebig}[1]{\tracesym \big( #1 \big)}

\newcommand{\A}{\mathcal{A}}           

\newcommand*{\trans}{{\mkern-1.5mu\mathsf{T}}}

\newcommand*{\T}{\mathbb{T}} 
\newcommand*{\R}{\mathbb{R}} 

\let\norm\relax
\DeclarePairedDelimiter{\normbracket}{\lVert}{\rVert}
\newcommand{\norm}{\normbracket}
\newcommand{\normbig}[1]{\big \lVert #1 \big \rVert}
\newcommand{\normBig}[1]{\Big \lVert #1 \Big\rVert}
\newcommand{\normbigg}[1]{\bigg \lVert #1 \bigg\rVert}

\newcommand{\hessian}{\operatorname{H}}

\newtheorem{theorem}{Theorem}

\newtheorem{algorithm}[theorem]{Algorithm}
\newtheorem{assumption}[theorem]{Assumption}
\newtheorem{remark}[theorem]{Remark}

\begin{document}

\title{Non-linear Gaussian smoothing with Taylor moment expansion}

\author{Zheng Zhao, \IEEEmembership{Student Member, IEEE} and Simo S\"{a}rkk\"{a}, \IEEEmembership{Senior Member, IEEE}
\thanks{This work was supported by Aalto ELEC doctoral school.}
\thanks{Zheng Zhao and Simo S\"{a}rkk\"{a} are with Department of Electrical Engineering and Automation, Aalto University, Finland (e-mail: zheng.zhao@aalto.fi).}}

\markboth{Journal of \LaTeX\ Class Files, Vol. 14, No. 8, August 2015}
{Shell \MakeLowercase{\textit{et al.}}: Bare Demo of IEEEtran.cls for IEEE Journals}
\maketitle

\begin{abstract}
This letter is concerned with solving continuous-discrete Gaussian smoothing problems by using the Taylor moment expansion (TME) scheme. In the proposed smoothing method, we apply the TME method to approximate the transition density of the stochastic differential equation in the dynamic model. Furthermore, we derive a theoretical error bound (in the mean square sense) of the TME smoothing estimates showing that the smoother is stable under weak assumptions. Numerical experiments show that the proposed smoother outperforms a number of baseline smoothers.
\end{abstract}


\IEEEpeerreviewmaketitle

\section{Introduction}
\label{sec:intro}

The aim of this paper is to consider smoothing problems (see, e.g., \cite{Jazwinski1970,Sarkka2019}) on non-linear continuous-discrete state-space model of the form
\begin{equation}
	\begin{split}
		\diff X(t) &= a(X(t)) \diff t + b(X(t)) \diff W(t),\\
		X(t_0)&=X_0 \sim \mathrm{N}(x_0 \cond m_0, P_0),\\
		Y_k &= h(X_k) + \xi_k, 
		\label{equ:problem-formulation}
	\end{split}
\end{equation}
on a temporal domain $\T\coloneqq [t_0, \infty)$, where $X\colon \T \to \R^d$ is an It\^{o} process driven by a standard Wiener process $W\colon\T\to\R^s$ modelling the state of the system. 
The random variable $Y_k\in\R^{d_y}$ stands for the measurement of $X_k \coloneqq X(t_k)$ via a (non-linear) function $h\colon\R^d\to\R^{d_y}$ and with a Gaussian noise $\xi_k \sim \mathrm{N}(\xi_k \cond 0, \Xi_k)$. The drift $a\colon\R^d\to\R^d$ and the dispersion coefficient $b\colon \T\to\R^{d\times s}$ of the stochastic differential equation (SDE) above are assumed to be regular enough so that the SDE is well-defined in It\^{o}'s SDE construction (see, e.g., \cite[][Ch. 5]{Karatzas1991}). 

Suppose that at times $t_1\leq t_2\leq \cdots \leq t_T\in\T$ we have a set of measurement data $y_{1:T} \coloneqq \lbrace y_k\rbrace_{k=1}^T$. The aim is to approximate the so-called smoothing posterior densities using Gaussian densities:
%
%
%
\begin{equation}
	p(x_k \cond y_{1:T}) \approx \mathrm{N}(x_k \cond m^s_k, P^s_k),
	\label{equ:smoothing-post-gaussian}
\end{equation}
for all $k=1,2,\ldots, T$. The goal is thus to compute the means and covariances $\lbrace m^s_k, P^s_k\rbrace_{k=1}^T $ using the measurement data. This is known as a continuous-discrete Gaussian smoothing problem~\cite{Sarkka2013, Sarkka2019, SimoGFS2013, Jazwinski1970, Law2015}. 

There exists a plethora of methods that can be used to obtain the Gaussian approximated density in Equation~\eqref{equ:smoothing-post-gaussian}. When the functions $a$, $b$, and $h$ are linear, the approximation in Equation~\eqref{equ:smoothing-post-gaussian} becomes exact, and one can use (continuous-discrete) Kalman filters and smoothers to solve them exactly (or up to arbitrary precision) in closed form~\cite{Sarkka2019}. However, for non-linear models, the smoothing distributions are no longer exactly Gaussian, while it is possible to approximate them by using, for example, extended Kalman filters and smoothers (EKFSs) \cite{Jazwinski1970}, sigma-point filters and smoothers \cite{Sarkka2013, Sarkka2019}, particle filters and smoothers~\cite{Chopin2020, Bain2009}, or projection filters and smoothers~\citep{Brigo1998, Koyama2018}.

One major challenge in solving the filtering and smoothing problems for non-linear models lies in recursively propagating filtering and smoothing solutions through the SDE in Equation~\eqref{equ:problem-formulation}. To that end, we need to be able to compute the transition densities 
\begin{equation}
	p(x_k \cond x_{k-1}),
	\label{equ:trans-density}
\end{equation}
for $k=1,2,\ldots,T$ in order to carry out the computations. However, the transition density of the SDE in Equation~\eqref{equ:problem-formulation} is, in general, analytically intractable and it needs be approximated. In order to obtain an approximation to the transition density, one solution is to locally approximate the non-linear SDE by a linear SDE \cite{Ozaki1993}. Another commonly used approach is to discretise the SDE by using It\^{o}--Taylor expansions~\cite{Kloeden1992, Cubature2009, CDCKF2010} (e.g., Euler--Maruyama or Milstein's method). However, low-order It\^{o}--Taylor expansions are accurate only when the discretisation time step is sufficiently small, while high-order expansions are numerically efficient only when the dispersion coefficient $b$ is commutative \cite[][Ch. 10]{Kloeden1992} which is too  restrictive for many practical models.

The main contribution of this letter is to develop a Taylor moment expansion (TME) based Gaussian smoother for computing smoothing density approximations as in Equation~\eqref{equ:smoothing-post-gaussian}. In the approach, TME is used to approximate the transition density as follows:
\begin{equation}
	p(x_k \cond x_{k-1}) \approx \mathrm{N}\big(x_k \cond g^M(x_{k-1}), Q^M(x_{k-1})\big), \nonumber
\end{equation}
where $g^M$ and $Q^M$ are determined by the expansion so that they approximate the exact $\expec{X_k \cond x_{k-1}}$ and $\cov{X_k \cond x_{k-1}}$, respectively, up to an arbitrary order of precision $M$. This idea heavily depends on our previous work that applies the TME method to Gaussian filters~\cite{ZhaoTME2020}. In addition, we analyse the stability of the proposed smoother by theoretically showing that the smoother error is bounded in a mean square sense under weak assumptions on the SDE and the filtering estimates. The experimental results show that the proposed smoother outperforms a number of baseline smoothers that are commonly used in practice.

This letter is organised as follows. In Section~\ref{sec:tme-smoother}, we present the Taylor moment expansion (TME) based approach for solving the continuous-discrete Gaussian smoothing problem. In Section~\ref{sec:stability}, we derive the error bound for the smoothing estimates. Experimental results are presented in Section~\ref{sec:experiments}, followed by conclusions in Section~\ref{sec:epilogue}.

\section{Taylor moment expansion Gaussian smoother}
\label{sec:tme-smoother}
Let us suppose that the filtering densities for $k=1,2,\ldots,T$ have already been approximated as Gaussian $p(x_k \cond y_{1:k}) \approx \mathrm{N}(x_k \cond m^f_k, P^f_k)$, where the means and covariance are given by $\lbrace m^f_k, P^f_k\rbrace_{k=1}^T$. These approximations can be computed with many different Gaussian filters (e.g., EKFs or sigma-point filters)~\cite{Kazufumi2000, Wu2006, Wang2019, ZhaoTME2020, Cubature2009}. To be able to implement the standard non-linear Gaussian smoother (see, e.g., \cite{Sarkka2010, Sarkka2013}) for the problem, we need to be able to compute the following quantities:
\begin{equation}
	\begin{split}
		&\int x_{k+1} \, p(x_{k+1} \cond y_{1:k}) \diff x_{k+1}\\
		&\approx \int g(x_k) \, \mathrm{N}(x_k \cond m^f_k, P^f_k) \diff x_k\coloneqq m^-_{k+1},\\
		&\int x_{k+1} \, x_{k+1}^\trans \, p(x_{k+1} \cond y_{1:k}) \diff x_{k+1} - m^-_{k+1} \, (m^-_{k+1})^\trans\\
		&\approx \int \Big(Q(x_{k}) + g(x_k) \, \big(g(x_k)\big)^\trans\Big) \, \mathrm{N}(x_k \cond m^f_k, P^f_k) \diff x_k\\
		&\quad-m^-_{k+1} \, (m^-_{k+1})^\trans,\coloneqq P^-_{k+1},\nonumber
	\end{split}
\end{equation}
and
\begin{equation}
	\begin{split}
		&\int x_k \, x_{k+1}^\trans \, p(x_{k+1} \cond x_k) \, p(x_k \cond y_{1:k}) \diff x_k \diff x_{k+1}\\
		&\quad - m^f_k \, (m^-_{k+1})^\trans \\
		&\approx \int x_k \, \big(g(x_k)\big)^\trans \, \mathrm{N}(x_k \cond m^f_k, P^f_k) \diff x_k - m^f_k \, (m^-_{k+1})^\trans\\
		&\coloneqq D_{k+1},\nonumber
	\end{split}
\end{equation}
where $g(x_k) \coloneqq \expec{X_{k+1} \cond x_k}$ and $Q(x_k) \coloneqq \cov{X_{k+1} \cond x_k}$. Methods for approximating the functions/expectations $g$ and $Q$ include, for example, SDE linearisations~\cite{Ozaki1993} or It\^{o}--Taylor expansions (see Sec.~\ref{sec:intro}). In the following, we show how to use the Taylor moment expansion (TME) method~\cite{ZhaoTME2020} to approximate these quantities. 

Consider a target function $\phi \in C^{2 \, (M +1)}(\R^d;\R)$, where $C^M(\R^d;\R)$ denotes the space of $2\,(M+1)$ times continuously differentiable functions mapping from $\R^d$ to $\R$. Suppose that the SDE coefficients $a \in C^M(\R^d;\R^d)$ and $b \in C^M(\R^d;\R^{d\times s})$ are $M$ times differentiable. It is shown by~\cite{ZhaoTME2020, Zmirou1989, Kessler1997} that for every $t\geq s\in\T$,
\begin{equation}
	\expec{\phi(X(t)) \cond x_s} = \sum^M_{r=0} \frac{(t-s)^r}{r!} \A^r\phi(x_s) + R^{M}_\phi(t,s, x_s),\nonumber
\end{equation}
where operator $\A$ is defined by ($\nabla$ and $\hessian$ denote gradient and Hessian, respectively)
\begin{equation}
	\A\phi(x) = (\nabla_x \phi(x))^\trans \, a(x) + \frac{1}{2}\tracebig{b(x) \, b(x)^\trans \, \hessian_x\phi(x)},\nonumber
\end{equation}
and the remainder is given by
\begin{equation}
	R^{M}_{\phi}(t,s, x_s) = \int^t_s\int^{\tau_1}_{s}\cdots\int^{\tau_M}_{s} \expecbig{\A^{M+1} \phi(X(\tau)) \cond x_s} \diff \tau. \nonumber
\end{equation}
The remainder above is in general intractable, hence, we usually discard it in order to compute the TME numerically. If we do so, the differentiability conditions for $\phi$ and the SDE coefficients are then reduced to $2\, M$ and $M-1$, respectively.

Therefore, by choosing the identity function $\phi^\mathrm{I}(x) = x \in\R^{d}$, we can form an $M$th order TME approximation to $g$ by
\begin{equation}
	\begin{split}
		g(x_k) &\approx \sum^M_{r=0} \frac{1}{r!} \A^r\phi^\mathrm{I}(x_k) \, (t_{k+1}-t_k)^r \coloneqq g^M(x_k),\nonumber
	\end{split}
\end{equation}
where $\A$ is applied elementwise for any vector/matrix-valued target function. As for the covariance function $Q$, we need to introduce $\phi^{\mathrm{II}}(x) = x\,x^\trans \in\R^{d\times d}$. Based on this, one possible approximation to $Q$ is \cite{ZhaoTME2020}
\begin{equation}
	Q(x_k) \approx \sum^M_{r=1} \frac{1}{r!} \, \Phi_r(x_k) \, (t_{k+1}-t_k)^r \coloneqq Q^M(x_k),\nonumber
\end{equation}
where 
\begin{equation}
	\Phi_r(x_k) = \A^r\phi^{\mathrm{II}}(x_k) - \sum^r_{k=0} \binom{r}{k} \A^k\phi^\mathrm{I}(x_s) \big( \A^{r-k}\phi^\mathrm{I}(x_s) \big)^\trans. \nonumber
\end{equation}
\begin{remark}
	There is a fundamental difference between EKFSs and TME in the use of Taylor expansion. In EKFSs, the Taylor expansion is taken on the SDE coefficients, while in TME, the expansion is taken on the expectations (e.g., mean or covariance) with respect to the SDE distribution.
\end{remark}

Putting together all the results above, we define the TME Gaussian smoother in the following algorithm.
\begin{algorithm}[TME-$M$ Gaussian smoother]
	\label{alg:tme-smoother}
	Let $\lbrace m^f_k, P^f_k\rbrace_{k=1}^T$ be a collection of filtering means and covariances. The TME Gaussian smoother obtains $\lbrace m^s_k, P^s_k\rbrace_{k=1}^{T-1}$ by recursively computing 
	\begin{equation}
		\begin{split}
		m^-_{k+1} &= \int g^M(x_k) \, \mathrm{N}(x_k \cond m^f_k, P^f_k) \diff x_k, \\
		P^-_{k+1} &= \int \Big(Q^M(x_{k}) + g^M(x_k) \, \big(g^M(x_k)\big)^\trans\Big) \\
		&\qquad\times \mathrm{N}(x_k \cond m^f_k, P^f_k) \diff x_k -m^-_{k+1} \, (m^-_{k+1})^\trans,\\
		D_{k+1} &= \int x_k \, \big(g^M(x_k)\big)^\trans \, \mathrm{N}(x_k \cond m^f_k, P^f_k) \diff x_k \\
		&\quad- m^f_k \, (m^-_{k+1})^\trans,\\
		G_k &= D_{k+1} \, (P^-_{k+1})^{-1}, \\
		m^s_k &= m^f_k + G_k \, (m^s_{k+1} - m^-_{k+1}),\\
		P^s_k &= P^f_k + G_k \, (P^s_{k+1} - P^-_{k+1}) \, G_k^\trans,
		\label{equ:tme-smoother}
		\end{split}
	\end{equation}
	for $k=T-1, T-2, \ldots, 1$. The recursion should be started from $m^s_T\coloneqq m^f_T$ and $P^s_T\coloneqq P^f_T$ at $k=T$.
\end{algorithm}

Sigma-point methods (e.g., unscented transform~\cite{Julier2004}, Gauss--Hermite~\cite{Wu2006}, spherical cubature~\cite{Cubature2009}, or sparse-grid quadratures~\cite{Jia2012, Radhakrishnan2016}) can be used to numerically approximate the expectations in Algorithm~\ref{alg:tme-smoother}. Specifically, for any integrand $z\colon\R^d \to \R^{n}$, the sigma-point methods approximate Gaussian integrals of the form $\int z(x) \, \mathrm{N}(x \cond m, P) \diff x$ by 
\begin{equation}
	S_{m, P}(z)\coloneqq \sum^K_{i=1} w_i \, z(\chi^i_{m, P}),
	\label{equ:sigma-points}
\end{equation}
where $\lbrace \chi^i_{m, P} \coloneqq m + \sqrt{P} \, \beta_i \rbrace^K_{i=1}$ are the integration nodes, and $\lbrace w_i\geq0 \rbrace^K_{i=1}$ are the associated weights of the nodes. We can define \emph{sigma-point TME Gaussian smoothers} as methods where the integrals in Algorithm~\eqref{alg:tme-smoother} are computed via Equation~\eqref{equ:sigma-points}.

\section{Stability analysis}
\label{sec:stability}
The (sigma-point) TME Gaussian smoother defined via Algorithm~\ref{alg:tme-smoother} is not exact smoother due to the errors introduced by the Gaussian approximations, TME approximations, and numerical integrations. Because these errors accumulate in time, it is important to analyse if the TME smoother is stable in the sense that the mean square error is finite:
\begin{equation}
	\sup_{T\geq 1} \max_{1 \leq k \leq T}\expecbig{\norm{X_k - m^s_k}_2^2} < \infty, \label{equ:sup_max_err}
\end{equation}
where $\norm{\cdot}_2$ stands for the Euclidean and spectral norms for vectors and matrices, respectively. Please note that $m^s_k$ indeed depends on $T$, hence, the error above is a function of both $k$ and $T$. The Equation~\eqref{equ:sup_max_err} essentially means that the smoothing errors are bounded over $k$ for every number of measurements $T\geq 1$.

We show that as long as the filter is stable and the SDE coefficients are suitably regular, then the smoother is stable. In addition, the concluded result is independent of the choice of the measurement model in Equation~\eqref{equ:problem-formulation}. We use the following assumptions in order to obtain the main result in Theorem~\ref{thm:stability}.

\begin{assumption}[Stable filtering]
	\label{assump:stable-filter}
	There exist $c_f(k)$ and $c_P>0$ such that
	\begin{equation}
		\expecbig{\norm{X_k - m^f_k}_2^2} \leq c_f(k) < \infty,\nonumber
	\end{equation}
	for all $k \geq 1$, and that $\sup_{k\geq 1}\trace{P_k^f} \leq c_P<\infty$ almost surely.
\end{assumption}
\begin{assumption}[System regularity]
	\label{assump:system}
	There exists $c_{S}>0$ such that
	\begin{equation}
		\norm{g^M(x) - S_{m, P}(g^M)}_2^2 \leq \norm{\nabla_x g^M(x)}_2^2 \,\norm{x - m}_2^2 + c_{S}\trace{P}, \nonumber
	\end{equation}
	for all $x\in\R^d$, $m\in\R^d$, and $P\in\R^{d\times d}$.
	Furthermore, $\sup_{x\in\R^d} \norm{\nabla_x g^M(x)}_2^2 \leq c_g$, and $\lambda_{\mathrm{min}}(Q^M(x))\geq c_Q >0$, where $\lambda_{\mathrm{min}}$ denotes the minimum eigenvalue.
\end{assumption}
\begin{assumption}
	\label{assump:cross}
	There exists $\overline{c}$ such that $\sup_{k\geq 1}\expecbig{\norm{m^f_{k+1} - S_{m^f_k, P^f_k}(g^M)}_2^2} \leq \overline{c}<\infty$.
\end{assumption}
Assumption~\ref{assump:stable-filter} ensures that the filter result is stable. This is necessary for ensuring stable smoothing, as the smoother takes the filtering result as input and produces the smoothing results by amending the filtering results. It is also worth noting that Assumption~\ref{assump:stable-filter} is based on the standard definitions of stable filters (see, e.g., \cite{Kazufumi2000,Xiong2006,LiLi2012, Bonnabel2012, ZhaoTME2020, Toni2020} for the definitions, and models and filters that suffice the stability assumptions).

Assumption~\ref{assump:system} ensures that the sigma-point integration of the TME approximant $g^M$ does not deviate significantly from $g^M$ with respect to a Lipschitz constant $c_g$. This is a commonly used assumption in analysing the stability of sigma-point filters~\cite[Assump. 4.1]{Toni2020}. In addition, Assumption~\ref{assump:system} requires the TME covariance approximant $Q^M$ to be strictly positive definite. The Bene\v{s} model~\cite{ZhaoTME2020} is an example that verifies Assumption~\ref{assump:system}, where $a(x) = \tanh(x)$ and $b=1$.

Assumption~\ref{assump:cross} says that the distances between the sigma-point mean predictions and the filtering means are uniformly bounded. In other words, the sigma-point mean prediction of $m^f_k$ produced by Algorithm~\ref{alg:tme-smoother} should stay around $m^f_{k+1}$ for all $k\geq 1$. This assumption is needed because the residual between $m^f_k$ and $m^f_{k+1}$ is the key in producing $m^s_k$. Moreover, we argue that it is hard to relax this assumption unless Algorithm~\ref{alg:tme-smoother} is aware how the filtering means are obtained (i.e., the relationship between $m^f_k$ and $m^f_{k+1}$ is not known by the algorithm).

\begin{theorem}
	\label{thm:stability}
	Let $c_G \coloneqq c_P^2 \, c_\chi \, (c_g + c_{S}) \, / \,c_Q^2$, where $\sqrt{c_\chi} \coloneqq \sum^K_{i=1}w_i \, \norm{\beta_i}_2$. Suppose that Assumptions~\ref{assump:stable-filter}, \ref{assump:system}, and~\ref{assump:cross} are satisfied, and that $2 \, c_G < 1$. Then Algorithm~\ref{alg:tme-smoother} with the sigma-point integration in Equation~\eqref{equ:sigma-points} is such that
	\begin{equation}
		\begin{split}
			&\expecbig{\norm{X_k - m^s_k}_2^2} \\
			&\leq 
			\begin{cases}
				c_f(T), &k=T,\\
				2 \, (c_f(T-1) + c_G \, \overline{c}), & k=T-1,\\
				2 \, c_f(k) + \big(\frac{4 \, c_G}{1 - 2\, c_G} + (2\,c_G)^{T-k}\big) \, \overline{c}, & 1\leq k \leq T - 2.
			\end{cases}\nonumber
		\end{split}
	\end{equation}
\end{theorem}
\begin{proof}
	When $k=T$, by definition $m^s_T \coloneqq m^f_T$, we recover the filtering bound $c_f(T)$ at $T$. When $k\leq T-1$, by expanding Equation~\eqref{equ:tme-smoother} we arrive at 
	\begin{equation}
		X_k - m^s_k = X_k - m^f_k - G_k \, (m^s_{k+1} - m^-_{k+1}),
	\end{equation} 
	and hence, 
	\begin{equation}
		\expecbig{\norm{X_k - m^s_k}_2^2} \leq 2 \, \Big(c_f(k) + \expecbig{ \norm{G_k}_2^2 \, \norm{m^s_{k+1} - m^-_{k+1}}_2^2} \Big). \nonumber
	\end{equation}
	By Assumptions~\ref{assump:stable-filter} to~\ref{assump:cross} and triangle inequalities, we have that almost surely
	\begin{equation}
		\begin{split}
			&\norm{G_k}_2^2 \\
			&= \normbig{S_{m^f_k, P^f_k}\big( (x - m^f_k) \, (g^M(x) - m^-_{k+1})^\trans \big) \, (P^-_{k+1})^{-1}}_2^2\\
			&= \normbigg{\sum^K_{i=1} w_i \, \Big(\sqrt{P^f_k} \, \beta_i \, \big( g^M\big(\chi_{m^f_k, P^f_k}^i\big) - m^-_{k+1}\big)^\trans \Big) (P^-_{k+1})^{-1}}_2^2\\
			&\leq \bigg(\sum^K_{i=1} w_i \, \normBig{\sqrt{P^f_k} \, \beta_i \, \big(  g^M\big(\chi_{m^f_k, P^f_k}^i\big) - m^-_{k+1}\big)^\trans}_2\bigg)^2 \, / \,c_Q^2\\
			&\leq c_P \, c_\chi\, (c_g \, c_P + c_{S} \, c_P) \, / \,c_Q^2 = c_G,\nonumber
		\end{split}
	\end{equation}
	and that
	\begin{equation}
		\begin{split}
			&\expecbig{\norm{m^s_{k+1} - m^-_{k+1}}_2^2} \\
			&= \expecbig{\norm{m^f_{k+1} - m^-_{k+1} + G_{k+1} \, (m^s_{k+2} - m^-_{k+2})}_2^2}\\
			&\leq 2 \, \overline{c} + 2 \, c_G \, \expecbig{\norm{m^s_{k+2} - m^-_{k+2}}_2^2},\nonumber
		\end{split}
	\end{equation}
	when $k\leq T-2$. We can expand the recursion of $\expecbig{\norm{m^s_{k+1} - m^-_{k+1}}_2^2}$ up to $\norm{m^s_T - m^-_T}_2^2$ and get 
	\begin{equation}
		\begin{split}
			\expecbig{\norm{m^s_{k+1} - m^-_{k+1}}_2^2} &\leq 2 \, \overline{c} \sum^{T-k-2}_{j=0} (2 \, c_G)^j + (2 \, c_G)^{T-k-1} \, \overline{c}.\nonumber
		\end{split}
	\end{equation}
	Since $c_G<\frac{1}{2}$, we have $\sum^{T-k-2}_{j=0} (2 \, c_G)^j <\frac{1}{1 - 2 \, c_G}$. By putting together all the bounds above, we conclude the result. 
\end{proof}
Theorem~\ref{thm:stability} shows that the TME smoother in Algorithm~\ref{alg:tme-smoother} has a contractive (in polynomial rate) error bound (provided that the filtering error $c_f(k)$ has a finite bound). In other words, the error $X_k - m^s_k$ is an asymptotically stable process~\cite[][Def. 2]{Tarn1976}. This concludes Equation~\eqref{equ:sup_max_err}.

\section{Experiments}
\label{sec:experiments}
In this section, we test the TME smoother defined in Algorithm~\ref{alg:tme-smoother} on a simulated model and compare it against other commonly used baseline smoothers\footnote{The implementation codes (in both Python and Matlab) of the experiments can be found at \href{https://github.com/zgbkdlm/tmefs}{https://github.com/zgbkdlm/tmefs}.}. The test model we choose here is a Lorenz '63 system~\cite{Law2014, LiXuechen2020} defined by
\begin{equation}
	\begin{split}
		\diff 
		\begin{bmatrix}
			X^1(t)\\X^2(t)\\X^3(t)
		\end{bmatrix}
		&=
		\begin{bmatrix}
			\kappa \, (X^2(t) - X^1(t))\\
			X^1(t) \, (\lambda - X^3(t)) - X^2(t)\\
			X^1(t) \, X^2(t) - \mu \, X^3(t)
		\end{bmatrix}
		+ \sigma \diff W(t),\\
		Y_k &= X^1(t_k) + \xi_k, \quad \xi_k\sim\mathrm{N}(0, 2), \nonumber
	\end{split}
\end{equation}
where $\kappa = 10$, $\lambda=28$, $\mu=2$, $\sigma=5$, and $W$ is a three-dimensional standard Wiener process with a unit spectral density. This model is considered challenging for non-linear filters and smoothers because making accurate predictions from its SDE can be hard~\cite{Evensen2009}. We simulate (numerical) ground truth trajectories and their measurements uniformly at $100$ time instants $0.02, 0.04, \ldots, 2$~s using $10,000$ Euler--Maruyama integration steps between each measurement.

In order to verify the efficacy of TME smoothers, we run combinations of commonly used filters and smoothers. These include a continuous-discrete extended Kalman filter (EKF-RK4) and smoother (EKS-RK4) with 4th order Runge--Kutta integration~\cite[][Alg. 10.35]{Sarkka2019}, and 3rd order Gauss--Hermite sigma-point filters (GHFs) and smoothers (GHSs). In particular, we compare TMEs (in orders 2 and 3) against the Euler--Maruyama (EM) scheme in GHFs and GHSs in order to show the benefits of using the TME scheme in sigma-point Gaussian smoothing.

We evaluate the performance of the smoothers by computing their root mean square errors (RMSEs) with respect to the simulated true trajectories. The means and standard deviations of the RMSEs are computed from $1,000$ independent Monte Carlo (MC) simulations of the model and runs of the smoothers.

\begin{table}[t!]
	\caption{RMSEs (in mean and standard deviation) of the smoothing estimates in combinations of different filters and smoothers. Bold numbers represent the best in each of their rows.}
	\label{tab:rmses}
	\begin{tabular}{@{}ccllll@{}}
		\toprule
		\multicolumn{2}{l}{\multirow{2}{*}{RMSE (std.)}} & \multicolumn{1}{c}{\multirow{2}{*}{EKS-RK4}} & \multicolumn{3}{c}{GHS} \\ \cmidrule(l){4-6} 
		\multicolumn{2}{l}{} & \multicolumn{1}{c}{} & \multicolumn{1}{c}{EM} & \multicolumn{1}{c}{TME-2} & \multicolumn{1}{c}{TME-3} \\ \midrule
		\multicolumn{2}{c}{EKF-RK4} & 18.05 (4.34) & \textbf{4.86} (0.68) & 6.23 (1.10) & 6.18 (1.10) \\ \cmidrule(r){1-2}
		\multirow{3}{*}{GHF} & EM & 14.43 (3.52) & 5.02 (0.77) & 4.93 (0.85) & \textbf{4.82} (0.83) \\
		& TME-2 & 10.27 (2.04) & 5.78 (0.99) & 3.95 (0.53) & \textbf{3.94} (0.53) \\
		& TME-3 & 10.15 (1.93) & 5.76 (0.97) & 3.98 (0.53) & \textbf{3.92} (0.52) \\ \bottomrule
	\end{tabular}
\end{table}

The RMSE results are shown in Table~\ref{tab:rmses}. We see that for all filters (except EKF-RK4), the smoothers that use the TME scheme have the best mean RMESs and lowest standard deviations, followed by EM and EKS. Moreover, the 3rd order TME results in better RMSEs compared to those of the 2nd order TME. The combination of EKF-RK4 and EKS-RK4 is found to be the worst, while the smoother that uses EM turns out to be the best under the EKF-RK4 filter.

\begin{figure}[t!]
	\centering
	\includegraphics[width=.99\linewidth]{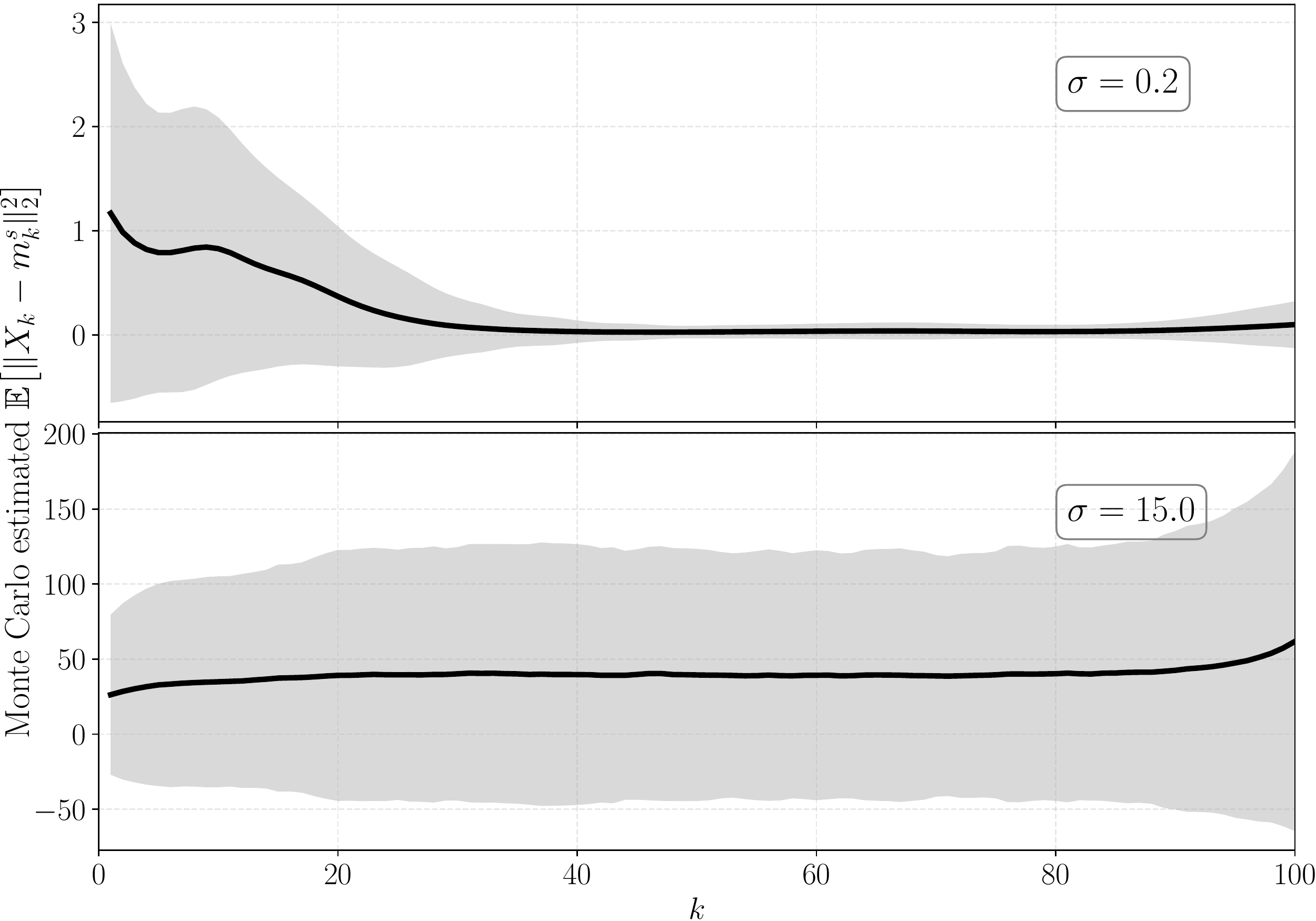}
	\caption{MC estimates of $\expecbig{\norm{X_k - m^s_k}_2^2}$. Shaded area stand for the 0.95 confidence interval of the MC samples.}
	\label{fig:error-bound-sigma}
\end{figure}

In order to verify the stability analysis in Theorem~\ref{thm:stability}, we estimate the error bound $\expecbig{\norm{X_k - m^s_k}_2^2}$ of the TME-2 filter and smoother by using $10,000$ MC runs. Since $c_G$ is a reciprocal of $c_Q^2$, we hereby check if the error bound tends to be contractive as $c_Q$ increases. This is done by tuning the $\sigma$ of the test model, as $c_Q$ is lower bounded by $\sigma^2$.

The error bounds with two different $\sigma$ settings are shown in Figure~\ref{fig:error-bound-sigma}. We see that when $\sigma=0.2$ is small (i.e., $c_Q$ is small and $c_G$ is large) the error bound increases as $k$ approaches to one. On the contrary, the error bound becomes contractive as $k$ goes to one when $\sigma=15$ is large (i.e., $c_Q$ is large and $c_G$ is small). This numerical result verifies the property of the theoretical error bound in Theorem~\ref{thm:stability} which shows that the bound is indeed increasing and contractive (as $k\to 1$) when the smoothing gain bound $c_Q$ is sufficiently large and small, respectively. 

\section{Conclusion}
\label{sec:epilogue}
We have proposed a Taylor moment expansion (TME) based Gaussian smoother for non-linear continuous-discrete state-space models. The key is to use TME to approximate the statistical properties (i.e., mean and covariance) of SDEs, resulting in asymptotically exact propagation of Gaussian smoothing estimates through the SDEs. Moreover, we have theoretically analysed the stability of the proposed TME smoothers, showing that the smoothers are stable if their preceding filtering procedures are stable and the SDE satisfies certain weak assumptions. The numerical experiments verify that the proposed TME smoothers outperform a number of commonly used smoothers, and that the theoretical error bound matches its numerical estimates in a Monte Carlo experiment.

\newpage

\bibliographystyle{IEEEtran}
\bibliography{refs}

\end{document}